\documentclass[12pt,bezier]{article}
\usepackage{amssymb}
\usepackage{mathrsfs}
\usepackage{amsmath}
\usepackage{amsfonts,amsthm,amssymb}
\usepackage{amsfonts}
\usepackage{graphics}

\newtheorem{lem}{Lemma}[section]
\newtheorem{thm}[lem]{Theorem}
\newtheorem{cor}[lem]{Corollary}

\newtheorem{rem}[lem]{Remark}

\newtheorem{Proposition}[lem]{Proposition}

\newtheorem{Conjecture}[lem]{Conjecture}

\begin{document}

\title{Minimum degree and size conditions for the proper connection number of graphs}

\author{ Xiaxia Guan$^a$, \quad Lina Xue$^a$, \quad Eddie Cheng$^b$, \quad Weihua Yang$^a$\footnote{Corresponding author. E-mail: ywh222@163.com,~yangweihua@tyut.edu.cn}\\
\\ \small $^a$Department of Mathematics, Taiyuan University of
Technology,\\
\small  Taiyuan Shanxi-030024,
China\\
\small $^b$Department of Mathematics and Statistics, Oakland University, Rochester, MI 48309, USA}
\date{}
\maketitle

{\small{\bf Abstract:}  
An edge-coloured graph $G$ is called $properly$ $connected$ if every two vertices are connected by a proper path. The $proper$ $connection$ $number$ of a connected graph $G$, denoted by $pc(G)$, is the smallest number of colours that are needed in order to make $G$ properly connected. Susan A. van Aardt et al.
gave a sufficient condition for the proper connection number to be at most $k$ in terms of the size of graphs. In this note, 
our main result is the following, by adding a minimum degree condition: Let $G$ be a connected graph of order $n$, $k\geq3$. If $|E(G)|\geq \binom{n-m-(k+1-m)(\delta+1)}{2} +(k+1-m)\binom{\delta+1}{2}+k+2$, then $pc(G)\leq k$, where $m$ takes the value $t$ if $\delta=1$ and $\lfloor \frac{k}{\delta-1} \rfloor$ if $\delta\geq2$.  Furthermore, if $k=2$ and $\delta=2$, 
$pc(G)\leq 2$, except $G\in \{G_{1}, G_{n}\}$ ($n\geq8$), where $G_{1}=K_{1}\vee 3K_{2}$ and  $G_{n}$ is obtained by taking a complete graph $K_{n-5}$  and $K_{1}\vee (2K_{2}$) with an arbitrary vertex of $K_{n-5}$ and a vertex with $d(v)=4$ in $K_{1}\vee (2K_{2}$) being joined. If  $k=2$, $\delta \geq 3$,
we conjecture $pc(G)\leq 2$, where $m$ takes the value $1$ if $\delta=3$ and $0$ if $\delta\geq4$ in the assumption.


\section{Introduction}

All graphs in this work are simple, finite connected and undirected. We follow \cite{Boundy} for graph theoretical notations  not defined here. Let $G$ be a  connected graph, we denote by $c(G)$ the $circumference$ of $G$, i.e., the order of a longest cycle of $G$, and by $p(G)$ the $detour$ $number$ of $G$, i.e., the order of a longest path of $G$.

As the extension of proper colourings are motivated by rainbow connections of graphs, Andrews et al.\cite{Andrews} and, independently, Borozan et al.\cite{Borozan} introduced the concept of proper connections in graphs. An edge-coloured graph $G$ is called $rainbow$-$connected$\cite{Chartrand} if every two vertices are connected by a path whose edges have different colours.  The $rainbow$ $connection$ $number$ of a connected graph $G$, denoted by $rc(G)$, is the smallest number of colours that are needed in order to make $G$ rainbow connected. An easy observation is that if $G$ has $n$ vertices then $rc(G)\leq n-1$, since one may colour the edges of a given spanning tree of $G$ with different colours, and colour the remaining edges with one of the already used colours.

A path in an edge-coloured graph is called a $proper$ $path$ if no two adjacent edges of the path are colored with one same color. An edge-colored graph $G$ is called $properly$ $connected$ if every pair of distinct vertices of $G$ is connected by a path whose edges are properly coloured. For a connected graph $G$, the $proper$ $connection$ $number$ of $G$, denoted by $pc(G)$, is defined as the smallest number of colors that are needed in order to make $G$ properly connected.

The proper connection of graphs has the following application. When building a communication network of wireless signal towers, one fundamental requirement is that the network be connected. If there cannot be a direct connection between two towers $A$ and $B$, say for example if there is a mountain in between, there must be a route through other towers to get from $A$ to $B$. As a wireless transmission passes through a signal tower, to avoid interference, it would be helpful if the incoming signal and the outgoing signal do not share the same frequency. Suppose that we assign a vertex to each signal tower, an edge between two vertices if the corresponding signal towers are directly connected by a signal and assign a color to each edge based on the assigned frequency used for the communication, then, the number of frequencies needed to assign to the connections between towers so that there is always a path avoiding interference between each pair of towers is precisely the proper connection number of the corresponding graph \cite{Huang}.

For the proper connection number of $G$, the following results are known.

\begin{Proposition}\label {Proposition 1.}
Let $G$ be a  connected graph of order $n$ (number of vertices) and size $m$ (number of edges). Then

(1) $1\leq pc(G)\leq \min\{\chi^{\prime}(G), rc(G)\}\leq m$, where $\chi^{\prime}(G)$ is the edge chromatic number of $G$,

(2) $pc(G) = 1$ if and only if $G = K_{n}$,

(3) $pc(G) = m$ if and only if $G \cong K_{1,m}$ is a star of size $m$,

(4) If $G$ is a tree, then $pc(G) = \Delta(G)$,

(5) If $G$ is traceable, i.e., there exists Hamilton path in $G$, then $pc(G)\leq 2$.

\end{Proposition}

For each pair of positive integers $n$ and $k$, we define  $g(n,k)$ to be  the smallest integer such that every connected graph of order $n$ and size at least $g(n,k)$ has proper connection number at most $k$. Huang, Li, and wang \cite{Wang} showed that $g(n,k)=\binom{n-k-1}{2}+k+2$ for $k=2$, $n\geq14$, and for $k=3$, $n\geq14$. In this paper we consider the function $g(n,k)$ by adding a minimum degree condition.

The analogous problem for rainbow connections was introduced in \cite{Kemnitz} and results on that problem appeared in \cite{Przybylo,Kemnitz,Li 2,Li 3,I}.

\section{Auxiliary results}

We shall use the following result of Andrews et al. \cite{Andrews}.

\begin{lem}[\cite{Andrews}]\label{lem2.1}
If $G$ is a  connected graph and $H$ is a connected spanning subgraph of $G$, then $pc(G)\leq pc(H)$. In particular, $pc(G)\leq pc(T)$ for every spanning tree $T$ of $G$.
\end{lem}

In fact, Lemma~\ref{lem2.1} also states that the $proper$ $connection$ $number$ is monotonic under adding edges.

\begin{Proposition}[\cite{Borozan}]\label{Proposition3}
If a graph $G$ contains a vertex $v$ such that $d(v)\geq 2$ and $pc(G-v)\leq 2$, then $pc(G)\leq 2$.
\end{Proposition}

\begin{thm}[\cite{Borozan}]\label{thm1}
If $G$ is a $2$-connected graph, then $pc(G)\leq 3$.
\end{thm}

Huang et al. \cite{Wang} extended Theorem~\ref{thm1} as follows.

\begin{thm}[\cite{Wang}]\label{thm2}
If $G$ is a connected bridgeless graph, then $pc(G)\leq 3$.
\end{thm}

For a connected graph $G$ with bridges, Huang et al. \cite{Wang} described the following natural approach: Let $B \subseteq E(G)$ be the set of cut-edges of a graph $G$. Let $\mathcal{C}$ denoted the set of connected components of $G^{\prime}$ = $(V, E\backslash B)$. There are two types of elements in $\mathcal{C}$, singletons and connected bridgeless subgraphs of $G$. Let $\mathcal{S} \subseteq \mathcal{C}$ denoted the singletons and let $\mathcal{D}$ = $\mathcal{C} \setminus \mathcal{S}$. Each element of $\mathcal{S}$ is, therefore, a vertex, and each element of $\mathcal{D}$ is a connected bridgeless subgraph of $G$. Contracting each element of $\mathcal{D}$ to a vertex, we obtain a new graph $G^{*}$. It is easy to see that $G^{*}$ is a tree, and the edge set of $G^{*}$ is $B$. Using the above notations, we have the following result.


\begin{thm}[\cite{Wang}]\label{thm3}
If $G$ is a connected graph, then $pc(G) \leq \max \{3, \Delta(G^{*})\}$.
\end{thm}

\begin{thm}[\cite{Susan}]\label{thm3.1}
Let $G$ be a connected graph of order $n$ and $k\geq 2$. If $|E(G)|\geq \binom{n-k-1}{2} +k+2$, then $pc(G) \leq k$ except when $k=2$ and $G\in \{G^{\ast}_{1}, G^{\ast}_{2}\}$.
\end{thm}

The two attentional graphs in Theorem~\ref{thm3.1}. and the following.

Let $G^{\ast}_{1}=K_{1}\vee (2K_{1}+K_{2})$ and $G^{\ast}_{2}=K_{1}\vee (K_{1}+2K_{2})$ where $G+H=(V_{G}\cup V_{H},E_{G}\cup E_{H})$ is the disjoint union and $G\vee H=(V_{G}\cup V_{H},E_{G}\cup E_{H}\cup \{uv : u\in V_{G}, v\in V_{H}\})$ is the join of the graphs $G=(V_{G},E_{G})$ and $H=(V_{H},E_{H})$.


We shall repeatedly use the following identities.

\begin{Proposition}\label{Proposition2.5}
For every pair of positive integers $a$ and $b$,

$\binom{a}{2}+\binom{b}{2}$=$\binom{a+b}{2}-ab$ and
$\binom{a+1}{2}$=$\binom{a}{2}+a$.
\end{Proposition}

In addition, we need the following result.

\begin{thm}[\cite{Dirac}]\label{thm Dirac}
Let $G$ be a graph with $n$ vertices. If $\delta(G) \geq \frac{n-1}{2}$ then $G$ has a Hamiltonian path. Moreover, if $\delta(G)\geq\frac{n}{2}$, then $G$ has a Hamiltonian cycle. Also, if $\delta(G) \geq \frac{n+1}{2}$, then $G$ is Hamiltonian-connected.
\end{thm}

\begin{thm}[\cite{Susan}]\label{thm2.0}
Let $G$ be a connected graph of order $n$ and $t$ bridges, then  $|E(G)|\leq \binom{n-t}{2} +t$.
\end{thm}

The next lemma will be useful for the proof of our main result.

\begin{lem}\label{lem2.6}
Let $G$ be a connected graph of order $n$  with $t$ bridges and $\delta=\delta(G)$.
Then $|E(G)|\leq \binom{n-m-(t-m)(\delta+1)}{2} +(t-m)\binom{\delta+1}{2}+t,$
where if $t=0$, then $m=0$. If  $\delta=1$, then $m=t$. If $\delta \geq 2$ and $t\neq 0$, then $m=\lfloor \frac{t-1}{\delta-1} \rfloor$.
\end{lem}

\begin{proof}

It is easy to see that the result holds for $t=0$.
We assume $t\geq 1$, then $|\mathcal{C}|=t+1$. Let $C_{1}, \cdots, C_{t+1}$ be the $|\mathcal{C}|$ elements  and $n_{i}$ ($i=1, \cdots, t+1$) be the orders of $C_{i}$.  Then $|E(G)|=\sum\limits_{i=1}^{t+1}|E(C_{i})|+t\leq\sum\limits_{i=1}^{t+1}\binom{n_{i}}{2}+t$. We define a supergraph $\widetilde{G}$ of $G$ (that is, $G$ is subgraph of $\widetilde{G}$) by adding all the possible edges in each $C_i$ in $G-B$  (that is, each component of $\widetilde{G}-B$ is complete. If no confusion arises,  we also use $C_i$ in $\widetilde{G}$ to denote the complete subgraphs obtained from $C_i$ of $\mathcal{C}$).

Now we construct the graph $G^{\prime}$ from $\widetilde{G}$ if there are two components $C_k$ and $C_l$ satisfying  $1<n_{k} \leq n_{l}$. We move a vertex $v$ from $C_{k}$ to $C_{l}$, replace  $v$ with an arbitrary vertex in $V(C_k)\setminus v$ for the bridges incident to $v$,  add the edges between $v$ and the vertices in $C_{l}$, and delete the edges between $v$ and the vertices in $C_{k}$, where $v$ is not adjacent to the vertices of $C_{l}$ in  $\widetilde{G}$.
We claim that $|E(\widetilde{G})|\geq|E(G)|$, since
\begin{equation}\label{eq1}
\begin{split}
|E(G^{\prime})|&=
\sum_{\substack{i=1\\ \neq k,l }}^{t+1}\binom{n_{i}}{2}+\binom{n_{k}-1}{2}+\binom{n_{l}+1}{2}+t\\
&=\sum_{\substack{i=1\\ \neq k,l }}^{t+1}\binom{n_{i}}{2}+\binom{n_{k}}{2}-(n_{k}-1)+\binom{n_{l}}{2}+n_{l}+t\\
&=|E(\widetilde{G})|+n_{l}-n_{k}+1>|E(\widetilde{G})|\geq|E(G)|.
 \end{split}
 \end{equation}

It can be seen that if $\delta=1$, then we repeatedly  move vertices so that $\mathcal{D}$ has only one element, that is, $|E(G)|\leq \binom{n-t}{2}+t$ by (\ref{eq1}).

We next assume $\delta_{\widetilde{G}}\geq 2$. Without loss of generality, we let $C_{i},1\leq i\leq m$ be the elements of $\mathcal{S}$ and $C_{i}, m< i\leq t+1$ be the elements of $\mathcal{D}$, respectively. Note that, each $C_i, i\leq m$ has at least $\delta$ neighbors in $\mathcal{C}$, and there are at most $m-1$ edges between  $C_i$ in $\mathcal{S}$. Thus, we have $m\delta-(m-1)\leq t$, that is, $m\leq \frac{t-1}{\delta-1}$.

Suppose there exist $n_{a}$ in $\widetilde{G}$ so that $1<n_{a}<\delta+1$, $m< a\leq t+1$. Then every vertex is at least incident to $\delta+1-n_{a}$ bridges in $V(C_{a})$ as the minimum degree of $\widetilde{G}$ is at least $\delta$, that is, there is at least  $\delta$ bridges incident to $V(C_{a})$ due to $n_{a}(\delta+1-n_{a})>\delta$ for $1<n_{a}<\delta+1$. By (\ref{eq1}), the size of $\widetilde{G}$ is fewer than the size  of the graph obtained by  moving vertices of $C_{a}$ to $C_{b}$ ($n_{b}\geq \delta+1$) so that $n_{a}=1$. Thus, by (\ref{eq1}), if the size of $G$ is as large as possible, then $m=\lfloor \frac{t-1}{\delta-1} \rfloor$ and $n_{i}\geq \delta+1,m< i\leq t+1$.

Therefore, we conclude that $|E(G)|\leq \binom{n-m-(t-m)(\delta+1)}{2} +(t-m)\binom{\delta+1}{2}+t$.
\end{proof}

We end this section with some results on the existence of long cycles in graphs that will be used below. We first state a classic result of Erd\H{o}s and Gallai~\cite{Erdos}.

\begin{thm}[\cite{Erdos}]\label{thm4}
Let $G$ be a graph of order $n$ and circumference $c(G)$. If
$$|E(G)|> \frac{c}{2} (n-1),$$
then $c(G)> c$.
\end{thm}

We shall use Woodall's extension of the Erd\H{o}s-Gallai Theorem, which may be stated as follows.

\begin{thm}[\cite{Woodall}]\label{thm5}
Let $G$ be a graph of order $n=tm+r$, where $m\geq 1$, $t\geq 0$, and $1\leq r\leq m $. If
$$|E(G)|> t\binom{m+1}{2} +\binom{r}{2},$$
then $c(G)\geq m+2$ and $p(G)\geq m+3$.
\end{thm}

\begin{thm}[Ore's Theorem,\cite{Ore}]\label{thm5.1}
If $G$ is a graph of order $n\geq 3$ such that $d(u)+d(v)\geq n$ for any pair $u$, $v$ of nonadjacent vertices in $G$, then $G$ is Hamiltonian.
\end{thm}

We shall use the following corollary of Ore's Theorem.

\begin{cor}
Assume $G$ is a graph of order $n$ and
$$|E(G)|\geq \binom{n-1}{2}+2.$$
Then $G$ is Hamiltonian.
\end{cor}

\section{Minimum degree and size conditions for the proper connection number of graphs}

We first establish an upper bound for the function $g(n,k)$ defined in Section $1$.

\begin{thm}\label{thm7}
Let $G$ be a connected graph of order $n$, $k\geq 3$. If $|E(G)|\geq \binom{n-m-(k+1-m)(\delta+1)}{2} +(k+1-m)\binom{\delta+1}{2}+k+2$, then $pc(G)\leq k$. Where if $\delta=1$, then $m=k+1$. If $\delta \geq 2$, then $m=\lfloor \frac{k}{\delta-1} \rfloor$.
\end{thm}

\begin{proof}
We first note that the statement is true for $\delta=1$ by Theorem~\ref{thm3.1}. In the following, we show that it is true for $\delta \geq 2$.

If $G$ is a connected bridgeless graph, then $pc(G)\leq 3 \leq k$ by Theorem~\ref{thm2}. So we assume that $G$ is connected with $t\geq 1$ bridges. We consider two cases.

{\bf Case 1.}\quad $t\leq k$.

In this case, $pc(G) \leq \max \{3, \Delta(G^{*})\} \leq \max\{3, t\} \leq \max\{3, k\} = k$ by Theorem~\ref{thm3}.

{\bf Case 2.}\quad $t\geq k+1$.

By Lemma~\ref{lem2.6}, we have $|E(G)|\leq \binom{n-m-(t-m)(\delta+1)}{2} +(t-m)\binom{\delta+1}{2}+t$, where $m=\lfloor \frac{t-1}{\delta-1} \rfloor$.

Suppose $\frac{t-1}{\delta-1}$ is not a positive integer (i.e., $\frac{t-2}{\delta-1}=\frac{t-1}{\delta-1}$), then
\begin{equation*} \label{eq:1}
\begin{split}
|E(G)|& \leq
\binom{n-m-(t-m)(\delta+1)}{2} +(t-m)\binom{\delta+1}{2}+t\\
&=\binom{n-m-(t-1-m)(\delta+1)}{2}-(\delta+1)(n-m-(t-m)(\delta+1))\\
&-\binom{\delta+1}{2}+(t-1-m)\binom{\delta+1}{2}+\binom{\delta+1}{2}+(t-1)+1\\
&=\binom{n-m-(t-1-m)(\delta+1)}{2} +(t-1-m)\binom{\delta+1}{2}\\
&+(t-1)-(\delta+1)(n-m-(t-m)(\delta+1))+1\\
&\leq \binom{n-m-(t-1-m)(\delta+1)}{2} +(t-1-m)\binom{\delta+1}{2}+(t-1)
 \end{split}
 \end{equation*}

 Suppose $\frac{t-1}{\delta-1}$ is  a positive integer (i.e., $\frac{t-2}{\delta-1}=\frac{t-1}{\delta-1}-1$), then
\begin{equation*} \label{eq:1}
\begin{split}
|E(G)|&\leq
\binom{n-m-(t-m)(\delta+1)}{2} +(t-m)\binom{\delta+1}{2}+t\\
&=\binom{n-(m-1)-(t-m)(\delta+1)}{2}-(n-m-(t-m)(\delta+1))\\
&+(t-m)\binom{\delta+1}{2}+t\\
&\leq \binom{n-(m-1)-(t-1-(m-1))(\delta+1)}{2}-(n-(m-1)\\
&-(t-1-(m-1))+(t-1-(m-1))\binom{\delta+1}{2}+(t-1)+1\\
&\leq \binom{n-m-(t-1-m)(\delta+1)}{2} +(t-1-m)\binom{\delta+1}{2}+(t-1)
 \end{split}
 \end{equation*}
Hence,   $|E(G)|$ is monotonic  decreasing on the bridges $t$, then if $t=k+1$, $|E(G)|$ is maximum and $|E(G)|\leq \binom{n-m-(k+1-m)(\delta+1)}{2} +(k+1-m)\binom{\delta+1}{2}+k+1$, a contradiction.
\end{proof}

\begin{rem}If $\delta\geq t$ (i.e., $m=0$), then  we construct a graph $G_{k}$ as follows: Take a complete graph $K_{n-(k+1)(\delta+1)}$ and $k+1$ complete graph $K_{\delta+1}$. Now a vertex of $K_{n-(k+1)(\delta+1)}$ and an arbitrary vertex of $(k+1)$ complete graph $K_{\delta+1})$ are joined. The resulting graph $G_{k}$ has order $n$, size $|E(G_{k})|=\binom{n-(k+1)(\delta+1)}{2} +(k+1)\binom{\delta+1}{2}+k+1$, and proper connection number $pc(G_{k})\geq k+1-m+m>k$. Therefore, $g(n,k)=\binom{n-(k+1)(\delta+1)}{2}+(k+1)\binom{\delta+1}{2}+k+2$, $k\geq 3$.
\end{rem}

Finally, we consider the case $k=2$. When $n\leq 2n+1$, then  $pc(G)\leq 2$ by Theorem~\ref{thm Dirac}.

\begin{thm}[\cite{Huang3}]\label{thm11}
Let $G$ be a connected noncomplete graph of order $5\leq n\leq 8$, if $G\notin\{G_{1},G_{8}\}$ and $\delta(G)\geq2$, then $pc(G)=2$.
\end{thm}

The two attentional graphs in Theorem \ref{thm11} and the following.

$G_{1}=K_{1}\vee 3K_{2}$ and  $G_{8}$ is obtained by taking a complete graph $K_{3}$  and $K_{1}\vee (2K_{2}$) with an arbitrary vertex of $K_{3}$ and a vertex with $d(v)=4$ in $K_{1}\vee (2K_{2}$) being joined.

\begin{thm}\label{thm8}
Let $G$ be a simple connected graph of order $n$ ($n\geq6$), $k=2$, $\delta=2$. If $|E(G)|\geq \binom{n-5}{2} +7$, then $pc(G)\leq 2$ unless $G\in \{G_{1}, G_{n}\}$,  where $G_{1}=K_{1}\vee 3K_{2}$ and  $G_{n}$ is obtained by taking a complete graph $K_{n-5}$  and $K_{1}\vee (2K_{2}$) with an arbitrary vertex of $K_{n-5}$ and a vertex with $d(v)=4$ in $K_{1}\vee (2K_{2}$) being joined.
\end{thm}

\begin{proof}

If $k=2$, $\delta=2$, then $m=2$. We observe that the requirement $|E(G)|\geq \binom{n-5}{2} +7$ is equivalent to the requirement $|E(\bar{G})|\leq 5n-22$ where $\bar{G}$ is the complement of $G$.

Now we suppose $6 \leq n\leq 8$, the result holds by Theorem \ref{thm11}.
%
%
%

We now assume that $n\geq 9$ and proceed by induction on $n$. We consider five cases.

{\bf Case 1.}\quad $p(G)=n-1$.

Let $P=v_{1} v_{2} \ldots v_{n-1}$ be a path of order $n-1$ in $G$ and let $w$ be the vertex not contained in $P$. As $\delta=2$, then $d(w)\geq 2$. It follows from Proposition~\ref{Proposition3} that $pc(G)\leq 2$, a contradiction.

\textbf{Claim.} If $|E(G)|\geq \binom{n-6}{2} +7$ ($n\geq9$) with order $n-1$ of $G$, and $pc(G)\leq 2$ unless $G_{n-1}$. Let $G^{\prime}=G+v$ and $d_{G^{\prime}}(v)=n-6$, then $pc(G^{\prime})\leq 2$ or  $G^{\prime}=G_{n}$.

\begin{proof} We suppose $G^{\prime}\neq G_{n}$, there exist $v_{i}$ which does not belong to  $V(K_{n-5})$ such that $vv_{i}\in E(G)$, then 
$p(G^{\prime})\geq n-1$. Hence $pc(G^{\prime})\leq 2$ according to Case 1 and Proposition~\ref{Proposition 1.}.
\end{proof}

{\bf Case 2.}\quad $p(G)=n-2$.

Let $P=v_{1} v_{2} \ldots v_{n-2}$ be a path of order $n-2$ in $G$ and let $w_{1}$, $w_{2}$ be the two vertices not contained in $P$. Since $\delta=2$, we have $d(w_{1})\geq 2$ and $d(w_{2})\geq 2$. Now assume $d_{p}(w_{2})\geq 2$ or $d_{p}(w_{2})\geq 2$,  without loss of generality, we assume $d_{p}(w_{2})\geq 2$. Since  $P$ is a longest path in $G$, then $w_{2}$ does not have consecutive neighbours on $P$ and  neither $v_{1}$ nor $v_{n-2}$ is a neighbour of $w_{2}$. And if $w_{1}v_{a}\in E(G)$, then $v_{1}v_{a+1}\in E(\bar{G})$. We have $d(w_{2})\leq \frac{n-5}{2}+1\leq n-6$ for $n\geq9$. Then $pc(G-w_{1})\leq2$  by our induction hypothesis. Thus, by applying Propositions~\ref{Proposition3}, we obtain that $pc(G)\leq2$.

Hence we may assume that $d_{p}(w_{1})=1$ and $d_{p}(w_{2})=1$, then $w_{1}$ and $w_{2}$ are adjacent. Let $w_{1}v_{a}, w_{2}v_{b}\in E(G)$, then $d(v_{1})\leq n-6$ and $d(v_{n-2})\leq n-6$, since $v_{1}v_{a+1}, v_{1}v_{a+2}, v_{1}v_{b+1}, v_{1}v_{b+2}$, $v_{1}w_{1}, v_{1}w_{2}, v_{n-2}v_{a-1}, v_{n-2}v_{a-2}, v_{n-2}v_{b-1}, v_{n-2}v_{b-2}, v_{n-2}w_{1}, v_{n-2}w_{2}, v_{1}v_{n-2}, v_{1}v_{n-3}, v_{2}v_{n-2}\in E(\bar{G})$. Then $|E(G-v)|\geq \binom{n-5}{2} +7-(n-6)=\binom{n-6}{2} +7$ ($v\in\{v_{1},v_{n-2}\}$) by Propositions~\ref{Proposition2.5}. If $\delta_{G-v}\geq 2$, then, by our induction hypothesis,  $pc(G-v)\leq 2$.  We obtain that $pc(G)=2$ by applying Propositions~\ref{Proposition3}. Thus we may assume $\delta_{G-v}=1$, then $d_{G}(v_{2})=d_{G}(v_{n-3})=d_{G}(w_{1})=d_{G}(w_{2})=2$ and $d_{G}(v_{1})> \frac{n}{2}$. There exist $v_{i}$ ($i$ is even) such that $v_{1}v_{i}\in E(G)$. Take a graph $G^{\prime}$ by adding  $v_{2}v_{i}$ in $G-v_{1}$. Note $\delta_{G^{\prime}}=2$, we have $G^{\prime}=G_{8}$ or $pc(G^{\prime})\leq 2$.  If  $pc(G^{\prime})\leq 2$, then $pc(G-v)\leq 2$ by our induction hypothesis and $pc(C)\leq2$, where $C=v_{1},\ldots,v_{i}$ is an even cycle. If $G^{\prime}=G_{8}$, then $p(G)\geq n-1$, a contradiction. Hence, $pc(G)\leq2$.



{\bf Case 3.}\quad $p(G)=n-3$.

Let $P=v_{1} v_{2} \ldots v_{n-3}$ be a path of order $n-3$ in $G$ and let $w_{1}$, $w_{2}$, $w_{3}$ be the three vertices not contained in $P$. 
 First prove that exist $v$, where $v\in\{w_{1}, w_{2}, w_{3}, v_{1}, v_{n-3}\}$, such that $\delta_{G-v}\geq 2$. Suppose $\delta_{G-v}=1$, for all $v\in\{w_{1}, w_{2}, w_{3}, v_{1}, v_{n-3}\}$, then $d(v_{2})=d(v_{n-4})=2$ and exist two vertices such that $d(v)=2$ for $v\in\{w_{1},w_{2},w_{3}\}$, without loss of generality, we assume $d(w_{1})=d(w_{2})=2$. In fact, if $w_{3}v_{a}\in E(G)$, then $v_{1}v_{a+1}\in E(\bar{G})$. Hence $d(w_{3})+d(v_{1})\geq n$. then $|E(\bar{G})|\geq 4(n-3)-8+n=5n-20$, a contradiction. Since $d(w_{i})\leq \frac{n-7}{2}+2<n-6$ for $i=1,2,3$  and $d(v_{1})$($d(v_{n-3})$)$\leq n-6$. Then $|E(G-v)|\geq \binom{n-5}{2} +7-(n-6)=\binom{n-6}{2} +7$ by Propositions~\ref{Proposition2.5}. Hence, by our induction hypothesis,  $pc(G-v)\leq 2$.  We obtain that $pc(G)\leq2$ by applying Propositions~\ref{Proposition3}.

{\bf Case 4.}\quad $p(G)=n-4$.

Let $P=v_{1} v_{2} \ldots v_{n-4}$ be a path of order $n-4$ in $G$ and let $w_{1}$, $w_{2}$, $w_{3}$, $w_{4}$ be the four vertices not contained in $P$.  We obtain that $pc(G)=2$, similar to Case 3.

{\bf Case 5.}\quad $p(G)\leq n-5$.

If $n=9$, $|E(G)|\geq 13$, so $G$ contains a cycle of length at least $3$. Since $p(G)\leq 4$, this implies that $c(G)=3$. Hence $|E(G)|\leq 4\binom{3}{2}=12$, a contradiction.

If $n=10$, $|E(G)|\geq 17$. By Thorem~\ref{thm5} (by taking $t=4$, $m=2$, $r=2$), we obtain $c(G)\geq 4$. Since $p(G)\leq 5$, this implies that $c(G)=4$ and the longest cycle can only exist one. Then it is easily seen that $pc(G)\leq 2$, a contradiction.

If $n=11$, then $|E(G)|\geq 21$. By Thorem~\ref{thm5} (by taking $t=3$, $m=3$, $r=2$), we obtain $c(G)\geq 5$. Since $p(G)\leq 6$, this implies that $c(G)=5$ and the longest cycle can only exist one. Then $|E(G)|\leq \binom{5}{2} +3\cdot 3=19<21$, a contradiction.

If $12\leq n\leq 20$, by the arguments similar to that of $n=11$, we have the contradicting on $|E(G)|\geq \binom{n-5}{2} +7$.

If $n>20$, then $|E(G)|\geq \binom{n-5}{2} +7=\binom{n-6}{2} +(n-6)+7>\binom{n-6}{2} +\binom{7}{2}$ and then it follows from Theorem~\ref{thm5} (by taking $t=1$, $m=n-7$, $r=7$) that $p(G)\geq m+3=n-4$, contradicting our assumption on $p(G)$.
\end{proof}



\begin{rem}
The way is not applicable for $k=2$ and $\delta \geq 3$, since $\delta_{G-v}=\delta_{G}-1$ and the vertex belonging to $P$ with the degree being $\delta_{G-v}$ ($\delta_{G}-1$) in $G-v$ ($G$) is possible and $|E(G)|$ reduces  along with the bridges $\delta$  growth. We conjecture the following.
\end{rem}

\begin{Conjecture}\label{Conjecture3.4}
 Let $G$ be a connected graph of order $n$, $k=2$ and $\delta \geq 3$. If $|E(G)|\geq \binom{n-m-(3-m)(\delta+1)}{2}+(3-m)\binom{\delta+1}{2}+4$, then $pc(G)\leq 2$, where $m$ takes the value $1$ if $\delta=3$ and $0$ if $\delta\geq4$. 
\end{Conjecture}

\begin{rem}
 We suggest a related (stronger) work for the conjecture above. Huang et al \cite{Huang3} showed  if $G$ is a connected noncomplete graph of order $n\geq 9$ and minimum degree $\delta(G)\geq n/4$, then $pc(G)=2$.
\end{rem}


\section{Acknowledgements}

The research is supported by NSFC (No.11671296) and NSF of Fujian Province (2015J05017), SRF for ROCS, SEM and Fund Program for the Scientific Activities of Selected
Returned Overseas Professionals in Shanxi Province.

\end{document}